\documentclass[a4paper,10pt]{article}
\usepackage{geometry}  
\usepackage{graphicx}
\usepackage{array}
\usepackage{amssymb}
\usepackage{amsmath}
\usepackage{amsthm}
\usepackage{arydshln}
\usepackage{graphicx}
\usepackage[parfill]{parskip} 
\usepackage[utf8]{inputenc}
\usepackage[english]{babel}
\usepackage{algorithm}
\usepackage{enumitem}
\usepackage{pgf}
\usepackage{tikz}
\usepackage{tikz-cd}
\usetikzlibrary{arrows,automata}
\usepackage[noend]{algpseudocode}
\usepackage{caption}
\usepackage[colorlinks=true,linkcolor=blue]{hyperref}
\usepackage{subcaption}
\usepackage{fancyhdr}
\usepackage{xfrac}
\usepackage[bottom]{footmisc}
\theoremstyle{definition}
\newtheorem{thm}{Theorem}[section]
\theoremstyle{definition}

\theoremstyle{definition}
\newtheorem{prop}[thm]{Proposition}
\theoremstyle{definition}
\newtheorem{dfn}[thm]{Definition}
\theoremstyle{definition}
\newtheorem{lem}[thm]{Lemma}
\theoremstyle{definition}

\theoremstyle{definition}

\graphicspath{ {images/} }

\begin{document}

\title{A note on the smooth blow-ups of $\mathbb{P}(1,1,1,k)$ in torus-invariant subvarieties}
\author{Daniel Cavey\footnote{Department of Mathematics and Statistics, Lancaster University, Lancaster, LA1 4YF, United Kingdom. Email address: cavey@lancaster.ac.uk}}
\date{}

\maketitle

\begin{abstract}
This papers classifies toric Fano 3-folds with singular locus $\{ \frac{1}{k}(1,1,1) \}$ for $k \in \mathbb{Z}_{\geq 1}$ building on the work of Batyrev \cite{ToroidalFano3-folds} and Watanabe-Watanabe \cite{WatanabeWatanabe}. This is achieved by completing an equivalent problem in the language of Fano polytopes. Furthermore we identify birational relationships between entries of the classification. For a fixed value $k \geq 4$, there are exactly two such toric Fano 3-folds linked by a blow-up in a torus-invariant line.
\end{abstract}

\section{Introduction}\label{Sec1}

A projective algebraic variety over $\mathbb{C}$ is Fano if the anticanonical divisor $-K_{X}$ is ample. Classifications of Fano varieties is an area of substantial interest in algebraic geometry. Most famously del Pezzo classified the smooth Fano varieties in dimension $2$, known as the 10 smooth del Pezzo surfaces. Mori-Mukai completed the classification of smooth Fano 3-folds finding 105 varieties \cite{ClassificationofFano3foldswithB2greaterequal2}.

Restricting the class of varieties in dimension $d$ to toric varieties, that is varieties with a suitable embedding of the algebraic torus $(\mathbb{C}^{*})^{d}$, allows for a combinatorial reinterpretation of the problem. The classification of toric $d$-dimensional Fano varieties is equivalent to the classification of specific $d$-dimensional lattice polytopes, known as Fano polytopes, up to a change of basis on the lattice.

A number of classifications of toric Fano varieties exist in the literature. Most notably Batyrev \cite{ToroidalFano3-folds} and Watanabe-Watanabe \cite{WatanabeWatanabe} simultaneously classified smooth toric Fano 3-folds. Batyrev further classified smooth toric Fano 4-folds \cite{OntheClassificationoftoricFano4folds}. Moving into the non-singular situation Kruezer-Skarke \cite{OntheClassificationofReflexivePolyhedra,ClassificationofReflexivePolyhedrainthreedimensions,CompleteClassificationofReflexivePolyhedrainfourDimensions} classify Gorenstein toric Fano varieties in dimensions $2,3$ and $4$. Kasprzyk classifies toric Fano 3-folds with at worst terminal/canonical singularities in \cite{CanonicalToricFanoThreefolds,ToricFanoThreefoldswithTerminalSingularities}. 

\begin{dfn}
Let $N \cong \mathbb{Z}^{n}$ be a lattice. A Fano polytope $P \subset N_{\mathbb{R}}=N \otimes \mathbb{R}$ is a full-dimensional convex polytope such that $\mathbf{0} \in \text{int}(P)$ and all vertices $v \in \mathcal{V}(P)$ have coprime coordinates.
\end{dfn}

The spanning fan $\Sigma_{P}$ of a Fano polytope $P$ gives rise to a toric variety $X_{P}$, leading to the previously mentioned equivalence for Fano variety classifications. Geometric properties of the toric variety can be seen at the level of the combinatorics of $P$. In particular one can observe the singularities of $X_{P}$; each maximal cone of $\Sigma_{P}$ describes a toric singularity on $X_{P}$.

\begin{dfn}
Consider the action of $\mu_{r}$, the cyclic group of order $r$, on $\mathbb{C}^{3}$ by 
\[ \epsilon \cdot (x,y,z) = (\epsilon^{a}x,\epsilon^{b}y, \epsilon^{c}z), \] 
where $\epsilon$ is an $r^{\text{th}}$ root of unity. The germ of the origin of $\text{Spec}(\mathbb{C}[x,y,z]^{\mu_{r}})$ is known as a $\frac{1}{r}(a,b,c)$ \emph{quotient singularity}. A \emph{cyclic quotient singularity} is a quotient singularity $\frac{1}{r}(a,b,c)$ such that $\gcd(r,a)=\gcd(r,b)= \gcd(r,c)=1$
\end{dfn}

Cyclic quotient singularities are toric, and the corresponding cone is simplicial. Of particular interest for this paper; (i) the simplicial cone whose generating rays form a basis of the lattice $N$, which describes a smooth patch $\mathbb{C}^{n} \subset X_{P}$, and is subsequently known as a smooth cone, and (ii) the cone over $\text{conv} \left\{ (1,0,0), (0,1,0), (-1,-1,-k) \right\}$, considered up to a change of basis, which describes a $\frac{1}{k}(1,1,1)$ singularity and is subsequently denoted $C_{\frac{1}{k}(1,1,1)}$.

As mentioned, Batyrev \cite{ToroidalFano3-folds} and Watanabe-Watanabe \cite{WatanabeWatanabe} classify smooth toric Fano 3-folds, and further identify any birational relations, that is blow-ups or blow-downs, within this classification. The resulting eighteen varieties are outlined in Table \ref{Table1} and fall into a cascade structure, a terminology coined by Reid--Suzuki \cite{CascadesOfProjectionsFromLogDelPezzoSurfaces}, rooted at $\mathbb{P}^{3}$ shown in Figure \ref{Fig1}. Here, and indeed throughout the paper, a blue line indicates a blow-up in a smooth point and a red line indicates a blow-up in a smooth torus-invariant  line.

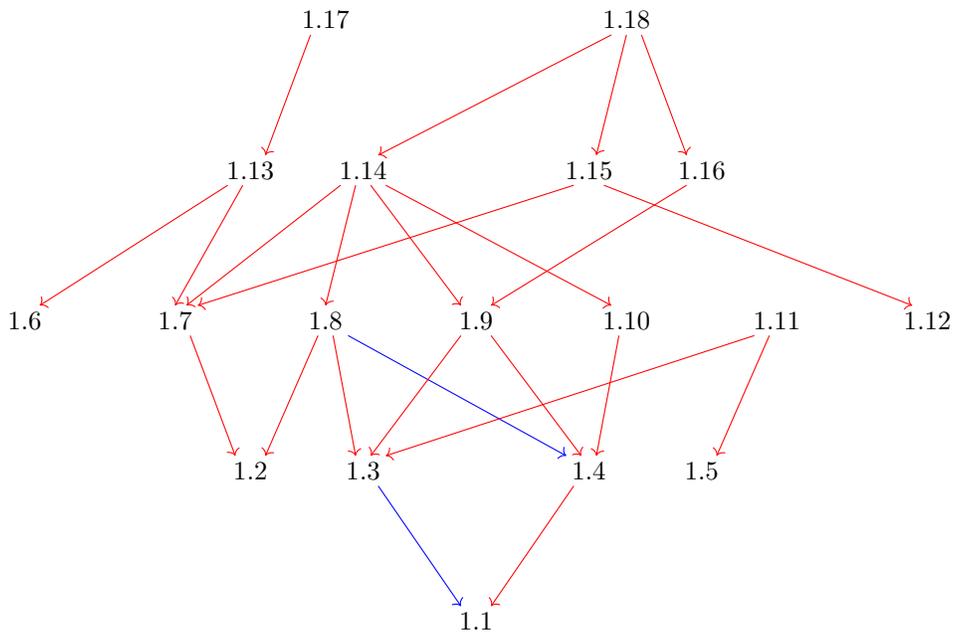
\begin{figure}[H]
\centering
\begin{tikzpicture}[transform shape]
\begin{scope}
\clip (-6.5,-0.3) rectangle (6.5cm,10.3cm); 

\node[] at (0,0) {$1.1$}; 

 \node[] at (-3,2) {$1.2$};
 \node[] at (-1.5,2) {$1.3$};
 \node[] at (1.5,2) {$1.4$};
 \node[] at (3,2) {$1.5$};
 
   \node[] at (-6,4) {$1.6$};
 \node[] at (-4,4) {$1.7$};
  \node[] at (-2,4) {$1.8$};
  \node[] at (0,4) {$1.9$};
 \node[] at (2,4) {$1.10$};
   \node[] at (4,4) {$1.11$};
 \node[] at (6,4) {$1.12$};
 
 \node[] at (-3,6) {$1.13$};
 \node[] at (-1.5,6) {$1.14$};
 \node[] at (1.5,6) {$1.15$};
 \node[] at (3,6) {$1.16$};
 
  \node[] at (-2,8) {$1.17$};
 \node[] at (2,8) {$1.18$};
 
 \path[every node/.style={font=\sffamily}]
 	(0.2,0.2) edge[<-, red] (1.3,1.8)
	(-0.2,0.2) edge[<-, blue] (-1.3,1.8)
	
	(-3.2,2.2) edge[<-, red] (-3.8,3.8)
	(-2.8,2.2) edge[<-, red] (-2.1,3.8)
	(-1.6,2.2) edge[<-, red] (-1.9,3.8)
	(1.2,2.2) edge[<-, blue] (-1.7,3.8)
	(-1.4,2.2) edge[<-, red] (-0.2,3.8)
	(1.4,2.2) edge[<-, red] (0.2,3.8)
	(1.6,2.2) edge[<-, red] (1.9,3.8)
	(-1.2,2.2) edge[<-, red] (3.7,3.8)
	(3.2,2.2) edge[<-, red] (3.9,3.8)
	
	(-5.8,4.2) edge[<-, red] (-3.3,5.8)
	(-4,4.2) edge[<-, red] (-3.1,5.8)
	(-3.85,4.2) edge[<-, red] (-1.8,5.8)
	(-2,4.2) edge[<-, red] (-1.6,5.8)
	(-0.2,4.2) edge[<-, red] (-1.4,5.8)
	(1.8,4.2) edge[<-, red] (-1.2,5.8)
	(-3.7,4.2) edge[<-, red] (1.3,5.8)
	(5.8,4.2) edge[<-, red] (1.7,5.8)
	(0.2,4.2) edge[<-, red] (2.8,5.8)
	
	(-2.8,6.2) edge[<-, red] (-2.2,7.8)
	(-1.3,6.2) edge[<-, red] (1.8,7.8)
	(1.6,6.2) edge[<-, red] (2,7.8)
	(2.8,6.2) edge[<-, red] (2.2,7.8);	
	
\end{scope}
\end{tikzpicture}
\caption{Cascade of varieties rooted at $\mathbb{P}^{3}$}\label{Fig1}
\end{figure}

Motivated by the above, one may ask how this cascade generalises if we were to replace $\mathbb{P}^{3}$ by $\mathbb{P}(1,1,1,k)$ for $k \geq 2$. To answer this question we classify the toric Fano 3-folds $V$ with singular locus $\text{Sing}(V)= \left\{ \frac{1}{k}(1,1,1) \right\}$ using the combinatorial language introduced.

The cases $k=2$ and $k=3$ are terminal and canonical respectively, and so the classifications are readily available via a quick search of the Graded Ring Database which stores work of Kasprzyk \cite{CanonicalToricFanoThreefolds,ToricFanoThreefoldswithTerminalSingularities}. The results are given in the following propositions.

\begin{prop}
Let $V$ be a toric Fano 3-fold, $\text{Sing}(V) = \left\{ \frac{1}{2}(1,1,1) \right\}$. Then $V$ is isomorphic to one of the eighteen varieties in Table \ref{Table2}. Furthermore these varieties admit birational relationships as illustrated in Figure \ref{Fig2}.

\begin{figure}[H]
\centering
\begin{tikzpicture}[transform shape]
\begin{scope}
\clip (-6.5,-0.3) rectangle (6.5cm,10.3cm); 

\node[] at (0,0) {$2.1$}; 

 \node[] at (-1.5,2) {$2.2$};
 \node[] at (1.5,2) {$2.3$};
 
   \node[] at (-5,4) {$2.4$};
 \node[] at (-3,4) {$2.5$};
  \node[] at (-1,4) {$2.6$};
 \node[] at (1,4) {$2.7$};
   \node[] at (3,4) {$2.8$};
 \node[] at (5,4) {$2.9$};
 
 \node[] at (-5,6) {$2.10$};
 \node[] at (-3,6) {$2.11$};
 \node[] at (-1,6) {$2.12$};
 \node[] at (1,6) {$2.13$};
 \node[] at (3,6) {$2.14$};
 \node[] at (5,6) {$2.15$};
 
  \node[] at (-1.5,8) {$2.16$};
 \node[] at (1.5,8) {$2.17$};
 
 \node[] at (0,10) {$2.18$};
 
 \path[every node/.style={font=\sffamily}]
 	(0.2,0.2) edge[<-, blue] (1.3,1.8)
	(-0.2,0.2) edge[<-, red] (-1.3,1.8)
	
	(-1.4,2.2) edge[<-, blue] (-1.2,3.8)
	(-1.3,2.2) edge[<-, red] (0.8,3.8)
	(1.2,2.2) edge[<-, red] (-0.8,3.8)
	(1.4,2.2) edge[<-, red] (1.2,3.8)
	(1.6,2.2) edge[<-, blue] (2.8,3.8)
	(1.8,2.2) edge[<-, red] (4.8,3.8)
	
	(-5,4.2) edge[<-, red] (-5,5.8)
	(-3,4.2) edge[<-, red] (-3,5.8)
	(-3.2,4.2) edge[<-, red] (-4.8,5.8)
	(-2.8,4.2) edge[<-, red] (-1.2,5.8)
	(1.6,2.2) edge[<-, blue] (2.8,3.8)
	(-1,4.2) edge[<-, red] (-1,5.8)
	(-0.8,4.2) edge[<-, red] (0.8,5.8)
	(1.2,4.2) edge[<-, red] (2.8,5.8)
	(3,4.2) edge[<-, red] (3,5.8)
	(3.2,4.2) edge[<-, blue] (4.8,5.8)
	(2.8,4.2) edge[<-, red] (1.2,5.8)
	(4.8,4.2) edge[<-, red] (-0.8,5.8)
	
	(-2.8,6.2) edge[<-, red] (-1.7,7.8)
	(1.2,6.2) edge[<-, red] (1.4,7.8)
	(2.8,6.2) edge[<-, blue] (1.6,7.8)
	(4.8,6.2) edge[<-, red] (1.8,7.8);	
	
\end{scope}
\end{tikzpicture}
\caption{Cascade of varieties rooted at $\mathbb{P}(1,1,1,2)$}\label{Fig2}
\end{figure}
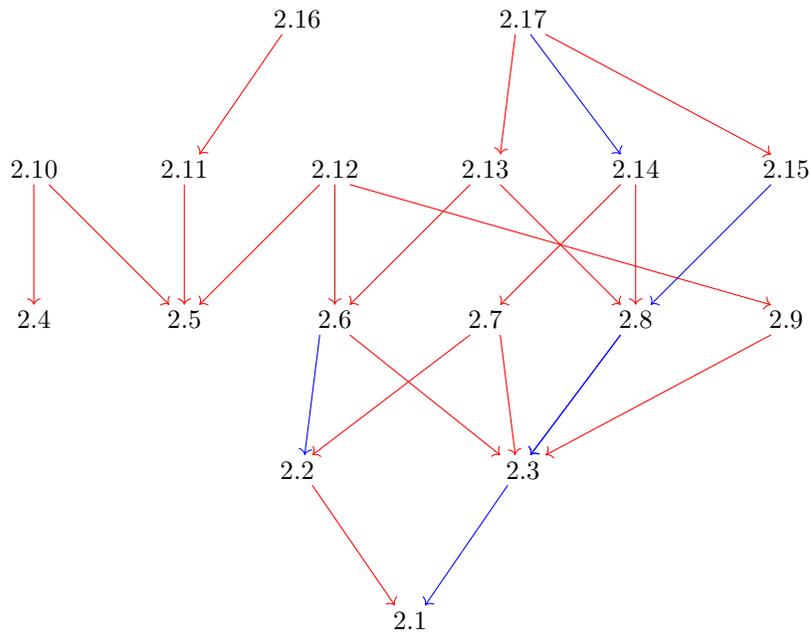
\end{prop}

\newpage

\begin{prop}
Let $V$ be a toric Fano 3-fold, $\text{Sing}(V) = \left\{ \frac{1}{3}(1,1,1) \right\}$. Then $V$ is isomorphic to one of the two varieties in Table \ref{Table3}. Furthermore these two varieties are related birationally, as illustrated in Figure \ref{Fig3}.
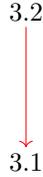
\begin{figure}[H]
\centering
\begin{tikzpicture}[transform shape]
\begin{scope}
\clip (-0.5,-0.3) rectangle (0.5cm,2.3cm); 

\node[] at (0,0) {$3.1$}; 
\node[] at (0,2) {$3.2$};
 
 \path[every node/.style={font=\sffamily}]
 	(0,0.2) edge[<-, red] (0,1.8);	
	
\end{scope}
\end{tikzpicture}
\caption{Cascade of varieties rooted at $\mathbb{P}(1,1,1,3)$}\label{Fig3}
\end{figure}
\end{prop}

The original material of this paper is then to deal with the case $k>3$, that is, when the varieties are no longer canonical. The results are summarised in Theorem \ref{1.3}

\begin{thm}\label{1.3}
For $k \geq 4$ let $V$ be a toric Fano 3-fold, $\text{Sing}(V) = \left\{ \frac{1}{k}(1,1,1) \right\}$. Then $V$ is isomorphic to one of the two varieties in Table \ref{Table4}. Furthermore these two varieties are related birationally, as illustrated in Figure \ref{Fig4}.
\begin{figure}[H]
\centering
\begin{tikzpicture}[transform shape]
\begin{scope}
\clip (-0.5,-0.3) rectangle (0.5cm,2.3cm); 

\node[] at (0,0) {k$.1$}; 
\node[] at (0,2) {k$.2$};
 
 \path[every node/.style={font=\sffamily}]
 	(0,0.2) edge[<-, red] (0,1.8);	
	
\end{scope}
\end{tikzpicture}
\caption{Cascade of varieties rooted at $\mathbb{P}(1,1,1,k)$}\label{Fig4}
\end{figure}
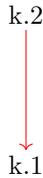
\end{thm}

In dimension 2, it is shown in \cite{DelPezzoSurfaceswithaSingle1k11singularity} that the $\mathbb{P}^{2}$ cascade generalises in some sense to the cascade of $\mathbb{P}(1,1,k)$ for $k>1$. This does not seem to be the case in the 3-fold setting. The majority of surfaces birationally related to $\mathbb{P}^{3}$, or indeed $\mathbb{P}(1,1,1,2)$ do not generalise, and so one could hypothesise that these are suitably interesting Fano 3-folds in a fashion analogous to the rich geometry of the low degree del Pezzo surfaces. Further it is worth remarking that variety 2.18 is the unique toric Fano 3-fold with singular locus $\left\{ \frac{1}{k}(1,1,1) \right\}$ which is not a smooth blow up of $\mathbb{P}(1,1,1,k)$ for any value of $k$, and is of particular intrigue.

The method outlined in this paper could be replicated if one wished to classify toric Fano 3-folds whose singular locus consists of any one fixed cyclic quotient singularity that is of particular interest to the reader.

\section{Proof of Theorem 1.5}

\begin{lem}\label{2.1}
Let $k>3$. Define the following sets of integer points:
\begin{align*}
U_{1}^{(k)} &:= \big\{ (-2,-1,-a), (-2,-1,-a+2), (-2,-1,-a+4), \ldots, (-2,-1,3) \big\}, \\
U_{2}^{(k)} &:= \big\{ (-1,-2,-a), (-1,-2,-a+2), (-1,-2,-a+4), \ldots, (-2,-1,3) \big\}, \\
U_{3}^{(k)} &:= \big\{ (-1,0,-b), (-1,0,-b+1), (-1,0,-b+2), \ldots, (-1,0,-1), (-1,0,1) \big\}, \\
U_{4}^{(k)} &:= \big\{ (-0,-1,-b), (0,-1,-b+1), (0,-1,-b+2), \ldots, (0,-1,-1), (0,-1,1) \big\}, \\
U_{5}^{(k)} &:= \big\{ (-1,1,-c), (-1,1,-c+2), (-1,1,-c+4), \ldots, (-1,1,k+3) \big\}, \\
U_{6}^{(k)} &:= \big\{ (1,-1,-c), (1,-1,-c+2), (1,-1,-c+4), \ldots, (1,-1,k+3) \big\}, \\
U_{7}^{(k)} &:= \big\{ (1,1,d), (1,1,d+1), (1,1,d+2), \ldots, (1,1,k-1), (1,1,k+1) \big\}, \\
U_{8}^{(k)} &:= \big\{ (1,2,e), (1,2,e+2), (1,2,e+4), \ldots, (1,2,2k+3) \big\}, \\
U_{9}^{(k)} &:= \big\{ (2,1,e), (2,1,e+2), (2,1,e+4), \ldots, (2,1,2k+3) \big\}, \\
U_{10}^{(k)} &:= \left\{ \begin{tabular}{c}
(-5,-4,-13), (-4,-5,-13), (-3,-2,-2k+1), (-2,-3,-2k+1), (-1,-1,-k+1), (-1,2,1), \\ (0,0,1), (0,1,1), (1,0,1), (1,3,k+1), (2,-1,1), (3,1,k+1), (4,-1,3), (5,1,7)
\end{tabular} \right\}, \\
U^{(k)} &:= U_{1}^{(k)} \cup U_{2}^{(k)} \cup \ldots \cup U_{10}^{(k)}.
\end{align*}
where
\[ a := \begin{cases} \frac{4}{3}k-1, & \text{if } k \equiv 0 \quad \text{(mod $3$)} \\ \frac{4}{3}(k-1)+1, & \text{if } k \equiv 1 \quad \text{(mod $3$)} \\ \frac{4}{3}(k-2)+1, & \text{if } k \equiv 2 \quad \text{(mod $3$)} \end{cases}, \]
\[ b := \begin{cases} \frac{2}{3}k-1, & \text{if } k \equiv 0 \quad \text{(mod $3$)} \\ \frac{2}{3}(k-1), & \text{if } k \equiv 1 \quad \text{(mod $3$)} \\ \frac{2}{3}(k-2)+1, & \text{if } k \equiv 2 \quad \text{(mod $3$)} \end{cases}, \]
\[ c := \begin{cases} \frac{1}{3}k+1, & \text{if } k \equiv 0 \quad \text{(mod $6$)} \\ \frac{1}{3}(k-1), & \text{if } k \equiv 1 \quad \text{(mod $6$)} \\ \frac{1}{3}(k-2)+1, & \text{if } k \equiv 2 \quad \text{(mod $6$)} \\ \frac{1}{3}(k-3), & \text{if } k \equiv 3 \quad \text{(mod $6$)} \\ \frac{1}{3}(k-4)-1, & \text{if } k \equiv 4 \quad \text{(mod $6$)} \\ \frac{1}{3}(k-5), & \text{if } k \equiv 5 \quad \text{(mod $6$)} \end{cases}, \]
\[ d := \begin{cases} \frac{1}{3}k+1, & \text{if } k \equiv 0 \quad \text{(mod $3$)} \\ \frac{1}{3}(k-1)+1, & \text{if } k \equiv 1 \quad \text{(mod $3$)} \\ \frac{1}{3}(k-2)+1, & \text{if } k \equiv 2 \quad \text{(mod $3$)} \end{cases}, \]
\[ e := \begin{cases} \frac{2}{3}k+1, & \text{if } k \equiv 0 \quad \text{(mod $3$)} \\ \frac{2}{3}(k-1)+1, & \text{if } k \equiv 1 \quad \text{(mod $3$)} \\ \frac{2}{3}(k-2)+3, & \text{if } k \equiv 2 \quad \text{(mod $3$)} \end{cases}. \]
If $P$ is a Fano polygon such that any cone of $P$ is either smooth or the $C_{\frac{1}{k}(1,1,1)}$ cone over $\text{conv} \{ (1,0,0),(0,1,0),(-1,-1,-k) \}$. Then
\[ \mathcal{V}(P) \backslash \big\{ (1,0,0),(0,1,0),(-1,-1,-k) \big\} \quad \subset \quad U^{(k)}. \]
\end{lem}

\begin{proof}
We aim to find all lattice points that could be added to the vertex set of $P$ that would not violate the assumptions on $P$. Convexity dictates that the plane through $(1,0,0), (0,1,0)$ and $(-1,-1,-k)$ defines a closed half-plane $kx+ky-3z \geq k$ of $N_{\mathbb{R}}$ in which all lattice points do not belong to $\mathcal{V}(P)$.

To sort through the remaining lattice points we define an algorithm. The algorithm is based on the fact that $P$ cannot contain  interior points or non-vertex boundary points other than those coming from $C_{\frac{1}{k}(1,1,1)}$. A lattice point $p \notin C_{\frac{1}{k}(1,1,1)}$ will be a non-vertex lattice point in $P$ if and only if $\mathcal{V}(P)$ contains a lattice point belonging to the polyhedral cone $C_{P}$ based at $p$ and generated by rays laying on the three lines $L_{1}, L_{2}, L_{3}$, where $L_{1}, L_{2}, L_{3}$ are the lines through $p$ and $(1,0,0), (0,1,0), (-1,-1,-k)$ respectively. So to determine all possible vertices we perform the following steps:
\begin{enumerate}[label=(\roman*)]
\item Define a set $A$ of the lattice points in the open half plane $kx+ky-3z<k$;
\item Pick a lattice point $p \in A$, and construct the cone $C_{p}$ described above;
\item Remove from $A$ all lattice points in $C_{p}$;
\item Return to step 2. and pick a point $p \in A$ that has not been chosen before;
\item Continue repeating until all points of $A$ have been chosen as $p$. 
\end{enumerate}

It remains to check that the algorithm does indeed terminate. It is enough to show that we can assume $A$ starts off as a finite set. To do this we use a bound, given by Hensley \cite{LatticeVertexPolytopeswithInteriorLatticePoints} and later improved by Lagarias-Ziegler \cite{BoundsforLatticePolytopescontainingaFixedNumberofInteriorPointsinaSublattice}, on the volume of a dimension $d$ polytope with $n>0$ interior points:
\[ \text{Vol}(P) \leq d! \cdot (8d)! \cdot 15^{d \cdot 2^{2d+1} \cdot n}. \]
Since $P$ is a $3$-dimensional polytope with at least one interior point, namely the origin, it follows that the volume is bounded by some $R \in \mathbb{Z}_{>0}$.

Consider $(x,y,z) \in A$. Define three lattice polytopes:
\begin{align*}
T_{1} :=& \text{conv} \left\{ (0,0,0), (1,0,0), (0,1,0), (x,y,z) \right\}, \\
T_{2} :=& \text{conv} \left\{ (0,0,0), (1,0,0), (-1,-1,-k), (x,y,z) \right\}, \\
T_{3} :=& \text{conv} \left\{ (0,0,0), (0,1,0), (-1,-1,-k), (x,y,z) \right\}. 
\end{align*}
If there exists $P$ with $(x,y,z) \in \mathcal{V}(P)$, then $T_{i} \subset P$, $\forall i$. It follows that the volume of $T_{i}$ is bounded by $R$. Calculate that
\begin{align*} 
\text{Vol}(T_{1}) &= \begin{vmatrix} 1 & 0 & 0 \\ 0 & 1 & 0 \\ x & y & z \end{vmatrix} = \lvert z \rvert \\
\text{Vol}(T_{2}) &= \begin{vmatrix} 1 & 0 & 0 \\ -1 & -1 & -k \\ x & y & z \end{vmatrix} = \lvert ky-z \rvert \\
\text{Vol}(T_{3}) &= \begin{vmatrix} 0 & 1 & 0 \\ -1 & -1 & -k \\ x & y & z \end{vmatrix} = \lvert kx-z \rvert
\end{align*}
It follows that $x,y,z$ are all bounded and so $A$ is finite. It is worth noting that the authors of \cite{LatticeVertexPolytopeswithInteriorLatticePoints,BoundsforLatticePolytopescontainingaFixedNumberofInteriorPointsinaSublattice} do not claim $R$ to be a sharp bound on the volume. Even if it was sharp, we do not believe $R$ would subsequently provide a sharp bound on the values that $x,y,z$ can take. However it is enough to assume that $A$ can be taken as finite in the algorthm and proves that the algorithm will indeed terminate. The set that is left after running the algorithm is a subset of the set $U^{(K)}$ described in the lemma for all $k \geq 4$.
\end{proof}

We are now ready to prove Theorem \ref{1.3}.

\begin{proof}"Theorem 1.5"
Consider the Fano polygon $P$ of a toric Fano 3-fold $V$, $\text{Sing}(V) = \left\{ \frac{1}{k}(1,1,1) \right\}$. Without loss of generality assume that the face of $P$ corresponding to the unique singularity is given by $\text{conv} \left\{ (1,0,0), (0,1,0), (-1,-1,-k) \right\}$. Since all other faces of $P$ must define a smooth cone, it follows from Lemma \ref{2.1} that $\mathcal{V}(P) \subset U^{(k)}$.

In particular there exists a vertex $v_{1} \in \mathcal{V}(P), v_{1} \neq (-1,-1,-k)$, such that $\text{conv} \{ (1,0,0), (0,1,0), v_{1} \}$ defines a face of $P$. This face will define a smooth cone and it must not determine that $\mathbf{0} \notin P$ by convexity. With this in mind define:
\[ L_{1}^{(k)} := \left\{ (x,y,z) \in U^{(k)} : \begin{vmatrix} 1 & 0 & 0 \\ 0 & 1 & 0 \\ x & y & z \end{vmatrix} = 1 \quad \text{and} \quad \begin{array}{c} \text{sign} ( n \cdot (-1,-1,-k) ) \neq \text{sign} (n \cdot v), \\ \text{where } n \text{ is the inward pointing normal} \\ \text{of conv}\{ (1,0,0), (0,1,0), (x,y,z) \} \end{array} \right\}. \] 
Define $v_{2} \neq (0,1,0)$ to be the vertex creating a face with $(1,0,0)$ and $(-1,-1,-k)$, and $v_{3} \neq (1,0,0)$ to be the vertex creating a face with $(0,1,0)$ and $(-1,-1,-k)$. Therefore consider two further similarly motivated sets of lattice points:
\[ L_{2}^{(k)} := \left\{ (x,y,z) \in U^{(k)} : \begin{vmatrix} 1 & 0 & 0 \\ -1 & -1 & -k \\ x & y & z \end{vmatrix} = 1 \quad \text{and} \quad \begin{array}{c} \text{sign} ( n \cdot (0,1,0) ) \neq \text{sign} (n \cdot v), \\ \text{where } n \text{ is the inward pointing normal} \\ \text{of conv}\{ (1,0,0), (-1,-1,-k), (x,y,z) \} \end{array} \right\}, \]
\[ L_{3}^{(k)} := \left\{ (x,y,z) \in U^{(k)} : \begin{vmatrix} 0 & 1 & 0 \\ -1 & -1 & -k \\ x & y & z \end{vmatrix} = 1 \quad \text{and} \quad \begin{array}{c} \text{sign} ( n \cdot (1,0,0) ) \neq \text{sign} (n \cdot v), \\ \text{where } n \text{ is the inward pointing normal} \\ \text{of conv}\{ (0,1,0), (-1,-1,-k), (x,y,z) \} \end{array} \right\}. \]

Using the list in the statement of Lemma \ref{2.1}, we can calculate $L_{1}, L_{2}, L_{3}$ explicitly:
\[ L_{1}^{(k)} = \left\{ \begin{array}{c}
(-2,-1,1), (-1,-2,1), (-1,0,1), (-1,1,1), (-1,2,1), (0,-1,1) \\
(0,0,1), (0,1,1), (1,-1,1), (1,0,1), (2,-1,1)
\end{array} \right\}, \]
\[ L_{2}^{(k)} = \left\{ \begin{array}{c}
(-3,-2,-2k+1), (-2,-1,-k+1), (-1,-1,-k+1), (-1,0,1), (-1,1,k+1), \\
(0,0,1), (1,0,1), (1,1,k+1), (1,2,2k+1), (2,1,k+1), (3,1,k+1)
\end{array} \right\}, \]
\[ L_{3}^{(k)} = \left\{ \begin{array}{c}
(-2,-3,-2k+1), (-1,-2,-k+1), (-1,-1,-k+1), (0,-1,1), (0,0,1), (0,1,1) \\
(1,-1,k+1), (1,1,k+1), (1,2,k+1), (1,3,k+1), (2,1,2k+1)
\end{array} \right\}. \]
It was worth noting that while choices for each of the $v_{i}$ must be made, it is not necessarily true that $v_{i} \neq v_{j}$ for $i \neq j$.

From here we construct sets $\mathcal{V}(P)$ defining the vertices of a suitable Fano polytope: 
\begin{enumerate}[label=(\roman*)]
\item Iterate through the possibilities for $v_{1}$;
\item Each choice for $v_{1}$ narrows down the possibilities for $v_{2}$ and $v_{3}$ from $L_{2}^{(k)}$ and $L_{3}^{(k)}$ respectively, by convexity;
\item Iterating through choices for $v_{2}$ narrows down the possibilities for $v_{3}$ from $L_{3}^{(k)}$;
\item Lattice points in $U^{(k)}$ which satisfy the three new convexity conditions can be added to $\mathcal{V}(P)$.
\end{enumerate} 
In attempting to do this a number of things can go wrong. For example, if there is a non-vertex lattice point in $\text{conv} \left\{ (1,0,0), (0,1,0), (-1,-1,-k), v_{1}, v_{2} \right\}$, other than those in $C_{\frac{1}{k}(1,1,1)}$, then $P$ would contain this lattice point and would therefore contain a second singular cone. This could also happen after adding $v_{3}$ or a vertex from $U^{(k)}$. Alternatively convexity requirements from adding $v_{1}$ and $v_{2}$ could leave us with no options for $v_{3}$, meaning a suitable $P$ cannot exist. Similarly it could happen that there a no possibilities in $U^{(k)}$ due to convexity from adding $v_{1}, v_{2}, v_{3}$ and the convex hull of the current set of vertices has a second singular cone meaning we cannot complete to the construction of a suitable polytope $P$. We demonstrate a sample computation for a particular choice of $v_{1}$.

Choose $v_{1}=(-2,-1,1) \in L_{1}^{(k)}$. It is worth noting that $(-2,-1,1) \notin L_{2}^{(k)}, L_{3}^{(k)}$ and so is not a suitable choice for $v_{2}$ or $v_{3}$. The new face $\text{conv} \left\{ (1,0,0), (0,1,0), (-2,-1,-1) \right\}$ bounds $P$ by the plane $x+y+4z < 1$. There are only three points in $L_{2}^{(k)}$ satisfying this bound and so are suitable choices of $v_{2}$, namely $(-3,-2,-2k+1), (-2,-1,-k+1)$ and $(-1,-1,-k+1)$. For $(-3,-2,-2k+1)$, note that $\text{conv} \left\{ (0,1,0), (-2,-1,1), (-3,-2,-2k+1) \right\}$ contains an interior point and so this is not a suitable choice for $v_{2}$. Similarly $\text{conv} \left\{ (-2,-1,-k+1), (-2,-1,1) \right\}$ contains interior points ruling out $v_{2} = (-2,-1,1)$. Therefore $v_{2}=(-1,-1,-k+1)$. Note that $v_{2} = (-1,-1,k+1) \in L_{3}^{(k)}$. Suppose initially that $v_{3} \neq (-1,-1,-k+1)$. Adding $v_{2}$ gave $P$ an additional bounding plane, $x-2y<1$, along with the pre-existing bound $x+y+4z<1$. No points in $L_{3}^{(k)}$ satisfy both these equations and so there would no possible choice for $v_{3}$. The only remaining choice for $v_{3}$ is $(-1,-1,-k+1)$. The polytope $\text{conv} \left\{ (1,0,0), (0,1,0), (-1,-1,-k), (-2,-1,1), (-1,-1,-k+1) \right\}$ has a singular cone over the face $\text{conv} \left\{ (1,0,0), (0,1,0), (-1,-1,-k+1) \right\}$, and it is necessary to add vertices from $U^{(k)}$ to change this. However we now have three bounding planes coming from adding $v_{1},v_{2}$ and $v_{3}$, namely $x+y+4z<1, x-2y<1$ and $y-2x<1$ respectively, and one can check that no points of $U^{(k)}$ satisfy all three of these bounds. Therefore there are no possible polytope constructions here.

The only suitable Fano polygons that are constructed through this method are the two that are listed in the statement of the theorem. It is routine to observe the blow up relation between the two varieties.
\end{proof}

\newpage

\begin{table}[H]\label{Table1}
\caption{Smooth toric Fano 3-folds $V$}\label{Table1}
\renewcommand{\arraystretch}{2}
\begin{center}
\begin{tabular}{| c | | c | c | c | c |} \hline
Id & $\mathcal{V}(P_{V})$ & $(-K_{X})^{3}$ & $\rho (X)$ & Model (where applicable) \\ \hline \hline 
1.1 & $\begin{matrix} 1 & 0 & 0 & -1 \\ 0 & 1 & 0 & -1 \\ 0 & 0 & 1 & -1 \end{matrix}$   & 64 & 1 & $\mathbb{P}^{3}$ \\ \hline
1.2 & $\begin{matrix} 1 & 0 & 0 & -1 & 0 \\ 0 & 1 & 0 & -1 & 0 \\ 0 & 0 & 1 & 0 & -1 \end{matrix}$ & 54 & 2 & $\mathbb{P}^{2} \times \mathbb{P}^{1}$ \\ \hline
1.3 & $\begin{matrix} 1 & 0 & 0 & -1 & -1 \\ 0 & 1 & 0 & -1 & 0 \\ 0 & 0 & 1 & -1 & 0 \end{matrix}$ & 56 & 2 & $\mathbb{P} \left( \mathcal{O}_{\mathbb{P}^{2}} \oplus \mathcal{O}_{\mathbb{P}^{2}}(1) \right)$ \\ \hline
1.4 & $\begin{matrix} 1 & 0 & 0 & -1 & -1 \\ 0 & 1 & 0 & -1 & -1 \\ 0 & 0 & 1 & -1 & 0 \end{matrix}$ & 54 & 2 & $\mathbb{P} \left( \mathcal{O}_{\mathbb{P}^{1}} \oplus \mathcal{O}_{\mathbb{P}^{1}} \oplus \mathcal{O}_{\mathbb{P}^{1}} \right)$ \\ \hline
1.5 & $\begin{matrix} 1 & 0 & -1 & -1 & -1 \\ 0 & 0 & -1 & 0 & 1 \\ 0 & 1 & -1 & 0 & 0 \end{matrix}$  & 62 & 2 & $\mathbb{P} \left( \mathcal{O}_{\mathbb{P}^{2}} \oplus \mathcal{O}_{\mathbb{P}^{2}}(2) \right)$ \\ \hline
1.6 & $\begin{matrix} 1 & 0 & 0 & -1 & 0 & 0 \\ 0 & 1 & 0 & 0 & -1 & 0 \\ 0 & 0 & 1 & 0 & 0 & -1 \end{matrix}$ & 48 & 3 & $\mathbb{P}^{1} \times \mathbb{P}^{1} \times \mathbb{P}^{1}$ \\ \hline
1.7 & $\begin{matrix} 1 & 0 & 0 & -1 & -1 & 0 \\ 0 & 1 & 0 & -1 & 0 & 0 \\ 0 & 0 & 1 & 0 & 0 & -1 \end{matrix}$  & 48 & 3 & $DS_{8} \times \mathbb{P}^{1}$ \\ \hline
\end{tabular}
\end{center}`
\end{table}

\renewcommand{\arraystretch}{2}
\begin{center}
\begin{tabular}{| c | | c | c | c | c |} \hline
Id & $\mathcal{V}(P_{V})$ & $(-K_{X})^{3}$ & $\rho (X)$ & Model (where applicable) \\ \hline \hline 
1.8 & $\begin{matrix} 1 & 0 & 0 & -1 & -1 & 0 \\ 0 & 1 & 0 & -1 & -1 & 0 \\ 0 & 0 & 1 & -1 & 0 & -1 \end{matrix}$  & 46 & 3 &  \\ \hline
1.9 & $\begin{matrix} 1 & 0 & 0 & -1 & -1 & -1 \\ 0 & 1 & 0 & -1 & -1 & 0 \\ 0 & 0 & 1 & -1 & 0 & 0 \end{matrix}$ & 50 & 3 & $\mathbb{P} \left( \mathcal{O}_{DS_{8}} \oplus \mathcal{O}_{DS_{8}}(1) \right)$ \\ \hline
1.10 & $\begin{matrix} 1 & 0 & 0 & -1 & -1 & 1 \\ 0 & 1 & 0 & -1 & -1 & 1 \\ 0 & 0 & 1 & -1 & 0 & 0 \end{matrix}$ & 44 & 3 & $\mathbb{P} \left( \mathcal{O}_{\mathbb{P}^{1} \times \mathbb{P}^{1}} \oplus \mathcal{O}_{\mathbb{P}^{1} \times \mathbb{P}^{1}}(1,-1) \right)$ \\ \hline
1.11 & $\begin{matrix} 1 & 0 & 0 & -1 & -1 & -1 \\ 0 & 1 & 0 & -1 & 0 & 1 \\ 0 & 0 & 1 & -1 & 0 & 0 \end{matrix}$ & 50 & 3 &  \\ \hline
1.12 & $\begin{matrix} 1 & 0 & -1 & -1 & 0 & -1 \\ 0 & 1 & -1 & 0 & 0 & 0 \\ 0 & 0 & 0 & 0 & -1 & 1 \end{matrix}$   & 52 & 3 & $\mathbb{P} \left( \mathcal{O}_{\mathbb{P}^{1} \times \mathbb{P}^{1}} \oplus \mathcal{O}_{\mathbb{P}^{1} \times \mathbb{P}^{1}}(1,-1) \right)$ \\ \hline
1.13 & $\begin{matrix} 1 & 0 & 0 & -1 & -1 & 0 & 0 \\ 0 & 1 & 0 & -1 & 0 & -1 & 0 \\ 0 & 0 & 1 & 0 & 0 & 0 & -1 \end{matrix}$ & 42 & 4 & $DS_{7} \times \mathbb{P}^{1}$ \\ \hline
1.14 & $\begin{matrix} 1 & 0 & 0 & -1 & -1 & 0 & 1 \\ 0 & 1 & 0 & -1 & -1 & 0 & 1 \\ 0 & 0 & 1 & -1 & 0 & -1 & 0 \end{matrix}$ & 40 & 4 &  \\ \hline
\end{tabular}
\end{center}

\renewcommand{\arraystretch}{2}
\begin{center}
\begin{tabular}{| c | | c | c | c | c |} \hline
Id & $\mathcal{V}(P_{V})$ & $(-K_{X})^{3}$ & $\rho (X)$ & Model (where applicable) \\ \hline \hline 
1.15 & $\begin{matrix} 1 & 0 & 0 & -1 & -1 & 0 & -1 \\ 0 & 1 & 0 & -1 & 0 & 0 & 0 \\ 0 & 0 & 1 & 0 & 0 & -1 & 1 \end{matrix}$  & 44 & 4 &  \\ \hline
1.16 & $\begin{matrix} 1 & 0 & 0 & -1 & -1 & -1 & 0 \\ 0 & 1 & 0 & -1 & -1 & 0 & -1 \\ 0 & 0 & 1 & -1 & 0 & 0 & 0 \end{matrix}$ & 46 & 4 & \\ \hline
1.17 & $\begin{matrix} 1 & 0 & 0 & 1 & -1 & 0 & 0 & -1 \\ 0 & 1 & 0 & 1 & 0 & -1 & 0 & -1 \\ 0 & 0 & 1 & 0 & 0 & 0 & -1 & 0 \end{matrix}$ & 36 & 5 & $DS_{6} \times \mathbb{P}^{1}$ \\ \hline
1.18 & $\begin{matrix} 1 & 0 & 0 & -1 & -1 & 0 & 1 & 1 \\ 0 & 1 & 0 & -1 & -1 & 0 & 1 & 1 \\ 0 & 0 & 1 & -1 & 0 & -1 & 0 & 1 \end{matrix}$ & 36 & 5 &  \\ \hline
\end{tabular}
\end{center}

\newpage

\begin{table}[H]\label{Table2}
\caption{Toric Fano 3-folds $V$, $\text{Sing}(V)= \left\{ \frac{1}{2}(1,1,1) \right\}$}\label{Table2}
\renewcommand{\arraystretch}{2}
\begin{center}
\begin{tabular}{| c | | c | c | c | c |} \hline
Id & $\mathcal{V}(P_{V})$ & $(-K_{X})^{3}$ & $\rho (X)$ & Id on GRDB \\ \hline \hline 
2.1 & $\begin{matrix} 1 & 0 & 0 & -1 \\ 0 & 1 & 0 & -1 \\ 0 & 0 & 1 & -2 \end{matrix}$ & $\frac{125}{2}$ & $1$ & 7 \\ \hline
2.2 & $\begin{matrix} 1 & 0 & 0 & -1 & -1 \\ 0 & 1 & 0 & -1 & -1 \\ 0 & 0 & 1 & -1 & -2 \end{matrix}$ & $\frac{101}{2}$ & $2$ & 44 \\ \hline
2.3 & $\begin{matrix} 1 & 0 & 0 & -1 & -1 \\ 0 & 1 & 0 & -1 & 0 \\ 0 & 0 & 1 & -2 & -1 \end{matrix}$ & $\frac{109}{2}$ & $2$ & 46 \\ \hline
2.4 & $\begin{matrix} 1 & 0 & -1 & -1 & -1 & 2 \\ 0 & 1 & -1 & 0 & 0 & 0 \\ 0 & 0 & -2 & -1 & 0 & 1 \end{matrix}$ & $\frac{113}{2}$ & $3$ & 122 \\ \hline
2.5 & $\begin{matrix} 1 & 0 & -1 & -1 & -1 & 1 \\ 0 & 1 & -1 & 0 & 0 & 0 \\ 0 & 0 & -2 & -1 & 0 & 1 \end{matrix}$  & $\frac{97}{2}$ & $3$ & 134 \\ \hline 
2.6 & $\begin{matrix} 1 & 0 & 0 & -1 & -1 & 1 \\ 0 & 1 & 0 & -1 & 0 & 0 \\ 0 & 0 & 1 & -2 & -1 & 1 \end{matrix}$ & $\frac{85}{2}$ & $3$ & 132 \\ \hline
2.7 & $\begin{matrix} 1 & 0 & 0 & -1 & -1 & -1 \\ 0 & 1 & 0 & -1 & -1 & 0 \\ 0 & 0 & 1 & -2 & -1 & -1 \end{matrix}$ & $\frac{93}{2}$ & $3$ & 137 \\ \hline
\end{tabular}
\end{center}
\end{table}

\newpage

\begin{table}[H]
\renewcommand{\arraystretch}{2}
\begin{center}
\begin{tabular}{| c | | c | c | c | c |} \hline
Id & $\mathcal{V}(P)$ & $(-K_{X})^{3}$ & $\rho (X)$ & Id on GRDB \\ \hline \hline 
2.8 & $\begin{matrix} 1 & 0 & 0 & -1 & -1 & 0 \\ 0 & 1 & 0 & -1 & 0 & -1 \\ 0 & 0 & 1 & -2 & -1 & -1 \end{matrix}$ & $\frac{93}{2}$ & $3$ & 121 \\ \hline
2.9 & $\begin{matrix} 1 & 0 & 0 & -1 & -1 & -1 \\ 0 & 1 & 0 & -1 & 0 & 0 \\ 0 & 0 & 1 & -2 & -1 & 0 \end{matrix}$ & $\frac{97}{2}$ & $3$ & 128 \\ \hline
2.10 & $\begin{matrix} 1 & 0 & -1 & -1 & -1 & 1 & 2 \\ 0 & 1 & -1 & 0 & 0 & 0 & 0 \\ 0 & 0 & -2 & 0 & -1 & 1 & 1 \end{matrix}$ & $\frac{81}{2}$ & $4$ & 253 \\ \hline
2.11 & $\begin{matrix} 1 & 0 & -1 & -1 & -1 & 0 & 1 \\ 0 & 1 & -1 & 0 & 0 & -1 & 0 \\ 0 & 0 & -2 & 0 & -1 & -1 & 1 \end{matrix}$ & $\frac{85}{2}$ & $4$ & 283 \\ \hline
2.12 & $\begin{matrix} 1 & 0 & 0 & -1 & -1 & -1 & 1 \\ 0 & 1 & 0 & -1 & 0 & 0 & 0 \\ 0 & 0 & 1 & -2 & 0 & -1 & 1 \end{matrix}$ & $\frac{73}{2}$ & $4$ & 262 \\ \hline
2.13 & $\begin{matrix} 1 & 0 & 0 & -1 & -1 & 0 & 1 \\ 0 & 1 & 0 & -1 & 0 & -1 & 0  \\ 0 & 0 & 1 & -2 & -1 & -1 & 1 \end{matrix}$  & $\frac{77}{2}$ & $4$ & 280 \\ \hline
2.14 & $\begin{matrix} 1 & 0 & 0 & -1 & -1 & 0 & -1 \\ 0 & 1 & 0 & -1 & 0 & -1 & -1 \\ 0 & 0 & 1 & -2 & -1 & -1 & -1 \end{matrix}$ & $\frac{85}{2}$ & $4$ & 209 \\ \hline
\end{tabular}
\end{center}
\end{table}

\newpage

\begin{table}[H]
\renewcommand{\arraystretch}{2}
\begin{center}
\begin{tabular}{| c | | c | c | c | c |} \hline
Id & $\mathcal{V}(P)$ & $(-K_{X})^{3}$ & $\rho (X)$ & Id on GRDB \\ \hline \hline 
2.15 & $\begin{matrix} 1 & 0 & 0 & -1 & -1 & 0 & 1 \\ 0 & 1 & 0 & -1 & 0 & -1 & 1 \\ 0 & 0 & 1 & -2 & -1 & -1 & 1 \end{matrix}$ & $\frac{77}{2}$ & $4$ & 170 \\ \hline
2.16 & $\begin{matrix} 1 & 0 & -1 & -1 & -1 & 0 & 1 & 1 \\ 0 & 1 & -1 & 0 & 0 & -1 & 0 & 1 \\ 0 & 0 & -2 & 0 & -1 & -1 & 1 & 1 \end{matrix}$ & $\frac{73}{2}$ & $5$ & 394 \\ \hline
2.17 &  $\begin{matrix} 1 & 0 & 0 & -1 & -1 & 0 & -1 & 1 \\ 0 & 1 & 0 & -1 & 0 & -1 & -1 & 1 \\ 0 & 0 & 1 & -2 & -1 & -1 & -1 & 1 \end{matrix}$  & $\frac{69}{2}$ & $5$ & 352 \\ \hline
2.18 & $\begin{matrix} 1 & 0 & -1 & 0 & 1 & -1 & 1 & -1 & 0 \\ 0 & 1 & -1 & -1 & -1 & 1 & 2 & -2 & -3 \\ 0 & 0 & -2 & 0 & 0 & 0 & 2 & -2 & -2 \end{matrix}$ & $\frac{69}{2}$ & $6$ & 514 \\ \hline
\end{tabular}
\end{center}
\end{table}

\newpage

\begin{table}[H]
\caption{Toric Fano 3-folds $V$, $\text{Sing}(V)= \left\{ \frac{1}{3}(1,1,1) \right\}$}\label{Table3}
\renewcommand{\arraystretch}{2}
\begin{center}
\begin{tabular}{| c | | c | c | c | c |} \hline
Id & $\mathcal{V}(P_{V})$ & $(-K_{X})^{3}$ & $\rho (X)$ & Id on GRDB \\ \hline \hline 
3.1 & $\begin{matrix} 1 & 0 & 0 & -1 \\ 0 & 1 & 0 & -1 \\ 0 & 0 & 1 & -3 \end{matrix}$ & $72$ & $1$ & 547377 \\ \hline
3.2 & $\begin{matrix} 1 & 0 & 0 & -1 & -1 \\ 0 & 1 & 0 & -1 & -1 \\ 0 & 0 & 1 & -3 & -2 \end{matrix}$ & $58$ & $2$ & 544337 \\ \hline
\end{tabular}
\end{center}
\end{table}

\bigskip

\begin{table}[H]
\caption{Toric Fano 3-folds $V$, $\text{Sing}(V)= \left\{ \frac{1}{k}(1,1,1) \right\}$}\label{Table4}
\renewcommand{\arraystretch}{2}
\begin{center}
\begin{tabular}{| c | | c | c | c |} \hline
Id & $\mathcal{V}(P_{V})$ & $(-K_{X})^{3}$ & $\rho (X)$ \\ \hline \hline 
k.1 & $\begin{matrix} 1 & 0 & 0 & -1 \\ 0 & 1 & 0 & -1 \\ 0 & 0 & 1 & -k \end{matrix}$ & $\frac{(k+3)^3}{k}$ & $1$ \\ \hline
k.2 & $\begin{matrix} 1 & 0 & 0 & -1 & -1 \\ 0 & 1 & 0 & -1 & -1 \\ 0 & 0 & 1 & -k & -k+1 \end{matrix}$ & $\frac{k^{3}+7k^{2}+35k+27}{k}$ & $2$ \\ \hline
\end{tabular}
\end{center}
\end{table}

\section*{Acknowledgements}

The author would like to thank Jonny Evans, his postdoctoral supervisor, for his guidance and insights throughout this project. This work was supported by Evans' EPSRC Grant EP/P02095X/2.

\end{document}